\documentclass[11pt,leqno]{article}
\usepackage[margin=1in]{geometry} 
\usepackage{amssymb,amsfonts,amsmath,bbm,mathrsfs,stmaryrd,mathtools,mathabx}
\usepackage{xcolor}
\usepackage{url}

\usepackage{extarrows}

\usepackage[shortlabels]{enumitem}
\usepackage{tensor}

\usepackage{xr}
\usepackage[T1]{fontenc}

\usepackage[colorlinks,
linkcolor=black!75!red,
citecolor=blue,
pdftitle={},
pdfproducer={pdfLaTeX},
pdfpagemode=None,
bookmarksopen=true,
bookmarksnumbered=true,
backref=page]{hyperref}

\usepackage{tikz}
\usetikzlibrary{arrows,calc,decorations.pathreplacing,decorations.markings,intersections,shapes.geometric,through,fit,shapes.symbols,positioning,decorations.pathmorphing}

\usepackage{braket}

\usepackage[amsmath,thmmarks,hyperref]{ntheorem}
\usepackage{cleveref}

\newcommand\nthalias[1]{\AddToHook{env/#1/begin}{\crefalias{lemma}{#1}}}

\nthalias{definition}
\nthalias{example}
\nthalias{examples}
\nthalias{remark}
\nthalias{remarks}
\nthalias{convention}
\nthalias{notation}
\nthalias{construction}
\nthalias{theoremN}
\nthalias{propositionN}
\nthalias{corollaryN}
\nthalias{lemma}
\nthalias{proposition}
\nthalias{corollary}
\nthalias{theorem}
\nthalias{conjecture}
\nthalias{question}
\nthalias{assumption}

\creflabelformat{enumi}{#2#1#3}

\crefname{section}{Section}{Sections}
\crefformat{section}{#2Section~#1#3} 
\Crefformat{section}{#2Section~#1#3} 

\crefname{subsection}{\S}{\S\S}
\AtBeginDocument{%
  \crefformat{subsection}{#2\S#1#3}%
  \Crefformat{subsection}{#2\S#1#3}%
}

\crefname{subsubsection}{\S}{\S\S}
\AtBeginDocument{%
  \crefformat{subsubsection}{#2\S#1#3}%
  \Crefformat{subsubsection}{#2\S#1#3}%
}

\theoremstyle{plain}

\newtheorem{lemma}{Lemma}[section]
\newtheorem{proposition}[lemma]{Proposition}
\newtheorem{corollary}[lemma]{Corollary}
\newtheorem{theorem}[lemma]{Theorem}

\theoremstyle{plain}
\theoremnumbering{Alph}
\newtheorem{theoremN}{Theorem}

\theoremstyle{plain}
\theorembodyfont{\upshape}
\theoremsymbol{\ensuremath{\blacklozenge}}

\newtheorem{remark}[lemma]{Remark}
\newtheorem{remarks}[lemma]{Remarks}

\newtheorem{notation}[lemma]{Notation}

\crefname{definition}{definition}{definitions}
\crefformat{definition}{#2definition~#1#3} 
\Crefformat{definition}{#2Definition~#1#3} 

\crefname{ex}{example}{examples}
\crefformat{example}{#2example~#1#3} 
\Crefformat{example}{#2Example~#1#3} 

\crefname{exs}{example}{examples}
\crefformat{examples}{#2example~#1#3} 
\Crefformat{examples}{#2Example~#1#3} 

\crefname{remark}{remark}{remarks}
\crefformat{remark}{#2remark~#1#3} 
\Crefformat{remark}{#2Remark~#1#3} 

\crefname{remarks}{remark}{remarks}
\crefformat{remarks}{#2remark~#1#3} 
\Crefformat{remarks}{#2Remark~#1#3} 

\crefname{convention}{convention}{conventions}
\crefformat{convention}{#2convention~#1#3} 
\Crefformat{convention}{#2Convention~#1#3} 

\crefname{notation}{notation}{notations}
\crefformat{notation}{#2notation~#1#3} 
\Crefformat{notation}{#2Notation~#1#3} 

\crefname{table}{table}{tables}
\crefformat{table}{#2table~#1#3} 
\Crefformat{table}{#2Table~#1#3}

\crefname{lemma}{lemma}{lemmas}
\crefformat{lemma}{#2lemma~#1#3} 
\Crefformat{lemma}{#2Lemma~#1#3} 

\crefname{proposition}{proposition}{propositions}
\crefformat{proposition}{#2proposition~#1#3} 
\Crefformat{proposition}{#2Proposition~#1#3} 

\crefname{propositionN}{proposition}{propositions}
\crefformat{propositionN}{#2proposition~#1#3} 
\Crefformat{propositionN}{#2Proposition~#1#3} 

\crefname{corollary}{corollary}{corollaries}
\crefformat{corollary}{#2corollary~#1#3} 
\Crefformat{corollary}{#2Corollary~#1#3} 

\crefname{corollaryN}{corollary}{corollaries}
\crefformat{corollaryN}{#2corollary~#1#3} 
\Crefformat{corollaryN}{#2Corollary~#1#3} 

\crefname{theorem}{theorem}{theorems}
\crefformat{theorem}{#2theorem~#1#3} 
\Crefformat{theorem}{#2Theorem~#1#3} 

\crefname{theoremN}{theorem}{theorems}
\crefformat{theoremN}{#2theorem~#1#3} 
\Crefformat{theoremN}{#2Theorem~#1#3} 

\crefname{enumi}{}{}
\crefformat{enumi}{#2#1#3}
\Crefformat{enumi}{#2#1#3}

\crefname{assumption}{assumption}{Assumptions}
\crefformat{assumption}{#2assumption~#1#3} 
\Crefformat{assumption}{#2Assumption~#1#3} 

\crefname{construction}{construction}{Constructions}
\crefformat{construction}{#2construction~#1#3} 
\Crefformat{construction}{#2Construction~#1#3} 

\crefname{question}{question}{Questions}
\crefformat{question}{#2question~#1#3} 
\Crefformat{question}{#2Question~#1#3} 

\crefname{equation}{}{}
\crefformat{equation}{(#2#1#3)} 
\Crefformat{equation}{(#2#1#3)}

\numberwithin{equation}{section}

\theoremstyle{nonumberplain}
\theoremsymbol{\ensuremath{\blacksquare}}

\newtheorem{proof}{Proof}

\newcommand\bC{{\mathbb C}}
\newcommand\bD{{\mathbb D}}

\newcommand\bR{{\mathbb R}}

\newcommand\bT{{\mathbb T}}

\newcommand\bZ{{\mathbb Z}}

\newcommand\cA{{\mathcal A}}

\newcommand\cC{{\mathcal C}}
\newcommand\cD{{\mathcal D}}

\newcommand\cL{{\mathcal L}}

\newcommand\cO{{\mathcal O}}

\DeclareMathOperator{\id}{id}

\DeclareMathOperator{\im}{\mathrm{im}}

\newcommand{\cat}[1]{\textsc{#1}}

\newcommand{\xrightarrowdbl}[2][]{%
  \xrightarrow[#1]{#2}\mathrel{\mkern-14mu}\rightarrow
}

\title{Local presentability and monadicity of forgetful functors\\ for operator algebraic categories}
\author{Alexandru Chirvasitu and Ian Thompson}

\begin{document}

\date{}

\newcommand{\Addresses}{{
  \bigskip
  \footnotesize

  \textsc{Department of Mathematics, University at Buffalo, Buffalo,
    NY 14260-2900, USA}\par\nopagebreak \textit{E-mail address}:
  \texttt{achirvas@buffalo.edu}

  \medskip

  \textsc{Department of Mathematical Sciences, University of Copenhagen, Universitetsparken 5, 2100 Copenhagen, Denmark}
    \par\nopagebreak
    \textit{E-mail address}: \texttt{ian@math.ku.dk}
  
}}

\maketitle

\begin{abstract}
In recent work of Lindenhovius and Zamdzhiev, it was established that the category of complete operator spaces, with completely contractive linear maps as morphisms, is locally countably presentable. In this work, we extend their conclusion to the non-complete setting and prove that the categories of operator systems, (Archimedean) order unit spaces, and unital operator algebras are all locally countably presentable as well. This is established through an analysis of forgetful functors and the identification of Eilenberg-Moore categories. We provide a complete understanding of adjunction and monadicity for forgetful functors between these categories, together with the categories of $C^*$-algebras, Banach spaces, and normed spaces. In addition, for various subcategories of function-theoretic objects, we investigate completeness and local presentability through Kadison's duality theorem.
\end{abstract}

\noindent {\em Key words: locally presentable category; monadic; operator space; operator system; non self-adjoint operator algebra; $C^*$-algebra; Banach and normed spaces; Eilenberg-Moore category}

\vspace{.5cm}

\noindent{MSC 2020: 18C35; 18C15; 46L07; 47L25; 47L30; 46A55; 47L75; 18C20

}


\section*{Introduction}

A category is said to be locally presentable if each object can be understood, in an appropriate sense, at the level of a colimit of appropriately small objects \cite{ar}. Locally presentable categories enjoy a significant number of regularity properties, such as completeness and cocompleteness, the existence of certain factorization systems, as well as being both well-powered and cowell-powered. As we shall see in this manuscript, local presentability also allows for a significant simplification of when a given functor is monadic. Among some of the most frequently encountered locally presentable categories are the categories of Banach spaces (or, alternatively, normed spaces) with contractive linear maps as morphisms, as well as the category of $C^*$-algebras, with $*$-homomorphisms as morphisms.

In this work, we establish local presentability for many categories that are often considered throughout functional analysis and mathematical physics. This is prompted by recent work by Lindenhovius and Zamdzhiev \cite{2412.20999v1}, where they established that the category of so-called complete operator spaces, together with completely contractive linear maps as morphisms, is locally $\aleph_1$-presentable. While Banach spaces and normed spaces can be identified with subspaces of commutative $C^*$-algebras, operator spaces are those linear spaces that are subspaces of possibly non-commutative $C^*$-algebras. To capture the ambient information defining an operator space, it is necessary that a description of objects is reflected in the structure of the matrix norms attached to the subspace (see \cite[Theorem 13.4]{pls-bk}).

Here, we extend Lindenhovius and Zamdzhiev's conclusion to non-complete operator spaces and establish local $\aleph_1$-presentability for a wealth of other categories of natural interest. One may consult \Cref{not:osos} and \Cref{se:osoa} for the relevant definitions.

\begin{theoremN}\label{th:A}
    The following categories are locally $\aleph_1$-presentable. \begin{enumerate}[(a),wide]
        \item The category $\cat{OSp}$ of (possibly non-complete) operator spaces with completely contractive linear maps as morphisms.
        \item The category $\cat{OSys}$ of operator systems with unital completely positive maps as morphisms.
        \item The category $\cat{(A)OU}$ of (Archimedean) order unit spaces, with positive unital maps as morphisms.
        \item The category $\cat{OAlg}$ of unital (non self-adjoint) operator algebras with unital completely contractive maps as morphisms.
    \end{enumerate}
\end{theoremN}

Operator systems and completely positive maps are not only foundational to the theory of $C^*$-algebras (see \cite{blk}, for example), but they also appear quite frequently throughout quantum physics. On the other hand, unital operator algebras play a recurring role in operator theory and, more specifically, within the theory of operator dilations \cite{pls-bk}.

Many of the proof techniques needed to establish \Cref{th:A} are derived from an analysis of monadic forgetful functors. Along the way, our arguments allow us to establish a conclusion of separate interest.

\begin{theoremN}\label{th:B}
    Consider the triple $(\cat{OSp}, \otimes_h, {\bf 1} = \mathbb{C})$ where $\otimes_h$ denotes the Haagerup tensor product of operator spaces. Then the triple $T = (\cat{OSp}, \otimes_h, \mathbb{C})$ is a monad and $\cat{OAlg}\simeq\cat{OSp}^T$ is the Eilenberg-Moore category for the monad.
\end{theoremN}

In \Cref{subse:fn.sys}, we turn our attention to the local presentability of subcategories formed through appropriate notions of commutativity. To this end, we consider the category $\cat{OSys}_c^\wedge\subset\cat{OSys}$ of complete operator subsystems of commutative $C^*$-algebras. Kadison proved that $\cat{OSys}_c^\wedge$ is contravariantly equivalent to the category of compact convex subsets of locally convex topological vector spaces, with continuous affine maps as morphisms \cite[Theorem II.1.8.]{alfsen2012compact}. Using Kadison's duality theorem and additional operator algebraic techniques, we are able to establish the following.

\begin{theoremN}\label{th:C}
    The following statements hold.\begin{enumerate}[(a),wide]
        \item The category $\cat{OSys}_c^\wedge$ is locally $\aleph_1$-presentable.
        \item Consider the subcategories $\cat{OSys}_s^\wedge$ and $\cat{OSys}_{bs}^\wedge$ of $\cat{OSys}_c^\wedge$, whose objects are contravariantly equivalent to (Bauer) simplices. Then $\cat{OSys}_s^\wedge$ and $\cat{OSys}_{bs}^\wedge$ do not have equalizers and, in particular, are not locally $\aleph_1$-presentable.
        \item Let $\cat{OSp}_c$ and $\cat{OSp}_c^\wedge$ denote the subcategories of $\cat{OSp}$ whose objects are (respectively, complete) operator spaces that are completely isometrically embeddable into a commutative $C^*$-algebra. Then $\cat{OSp}_c$ and $\cat{OSp}_c^\wedge$ are locally $\aleph_1$-presentable.
    \end{enumerate}
\end{theoremN}

With regard to \Cref{th:C} (b), we remark that the Bauer simplices are precisely those operator systems that are isomorphic to a commutative $C^*$-algebra (this follows from \cite[Corollary II.4.2.]{alfsen2012compact} and Gelfand's duality theorem). Despite the negative presentability result for this category, it is known that, in fact, the category of commutative $C^*$-algebras with $*$-homomorphisms as morphisms is locally $\aleph_1$-presentable \cite[Theorem 3.28]{ar}.

Finally, we determine the monadicity and existence of adjoints for forgetful functors between several categories from \Cref{th:A} and the categories of $C^*$-algebras, Banach spaces and Banach algebras, as well as normed spaces and normed algebras.

\subsection*{Acknowledgments}

We are grateful for instructive comments from V. Paulsen. The second author was partially supported by an NSERC Postdoctoral Fellowship.


\section{Adjunction and presentability for operator spaces and systems}\label{se.main}

\begin{notation}\label{not:osos}
  Categories that will make an appearance below include 
  \begin{itemize}[wide]
  \item $\cat{OSp}$, the category of \emph{operator spaces} (\cite[\S 13]{pls-bk}, \cite[\S 2.1]{er_os_2000}) with \emph{completely contractive} (or \emph{cc}) \cite[\S 2.1]{er_os_2000} linear maps as morphisms;
  \item $\cat{OSys}$, the category of \emph{operator systems} \cite[\S 13]{pls-bk} with unital \emph{completely positive} \cite[p.176]{pls-bk} (or \emph{ucp}) linear maps;
  \item $\cat{Norm}$ and $\cat{Ban}$, normed and Banach spaces respectively, with contractive maps as morphisms;
  \item $\cat{(A)OU}$, \emph{(Archimedean) order unit spaces} (\cite[\S 2]{kptt}, \cite[Definitions 2.4 and 2.7]{zbMATH05592601}), with positive unital maps as morphisms. 
  \end{itemize}
  For our work, the base field can be chosen as $\bR$ or $\bC$, and most of the discussion can be applied simultaneously. We occasionally make one of the choices to fix the notation.
  
  Unless otherwise specified (and by contrast to a number of sources, such as \cite[Definition 1.2]{pis_os}, \cite[\S 5.1]{er_os_2000} or \cite[Definition 4.2]{2412.20999v1}), we do not assume that such structures (e.g. operator spaces/systems) are complete. We mark categories of complete objects with `$\wedge$' superscripts. For example, $\cat{OSp}^{\wedge}$ denotes the category of complete operator spaces, as considered in \cite{2412.20999v1}.
\end{notation}

As noted in \cite[Proposition 13.3]{pls-bk} or \cite[discussion preceding Proposition 3.2]{kptt}, an operator system $(S,e)$ is also an operator space, with matrix norms $\|\cdot\|_n$ on $M_n(S)$ defined by
\begin{equation*}
  \|s\|_n\le 1
  \Longleftrightarrow
  \begin{pmatrix}
    e_n & s\\
    s^* & e_n
  \end{pmatrix}
  \ge 0
\end{equation*}Furthermore, the definition makes it clear that any unital completely positive map between operator systems is completely contractive and thus, we have a {\it forgetful} (or {\it inclusion}) functor
\begin{equation}\label{eq:osys2osp}
  \cat{OSys}\xrightarrow{\quad U\quad} \cat{OSp}. 
\end{equation}

\begin{theorem}\label{th:osos.fgt}
  The categories $\cat{OSys}$ and $\cat{OSp}$ are both complete and the inclusion functor $U$ is a right adjoint, so that $\cat{OSys}\subset \cat{OSp}$ is a full reflective subcategory. 
\end{theorem}
\begin{proof}
  For the completeness claim, it suffices (\cite[\S V.2, Theorem 1]{mcl} or \cite[Theorem 12.3]{ahs}) to describe the products and equalizers (checking that they indeed are products and equalizers is a simple task that we omit). This is carried out in \cite[Propositions 4.11 and 4.13]{2412.20999v1} for \emph{complete} operator spaces; we remind the reader of the constructions.
  
  \begin{enumerate}[(I), wide]
  \item \textbf{: Products.} The construction is standard in the theory of operator algebras: for a family $S_i$, $i\in I$ of operator systems (spaces), define the categorical product $\displaystyle \prod^\infty S_i$ to be the set of {\it bounded} tuples
    \begin{equation*}
      (s_i)_{i\in I}\in \prod_{i\in I}S_i
    \end{equation*}
    in the Cartesian product. For operator spaces `bounded' means the obvious, i.e. having a uniform bound on the norm of all of the $s_i$, whereas for operator systems the phrase means that there is some $r>0$ such that
    \begin{equation*}
      s_i\le re_i,\ \forall i\in I,
    \end{equation*}
    where $e_i\in S_i$ is the order unit. The matrix order and norm in the two cases (operator systems and spaces, respectively) is defined similarly: either way, for every positive integer $n$ we have
    \begin{equation*}
      M_n\otimes \prod^{\infty}_i S_i\cong \prod^{\infty}_iM_n(S_i).
    \end{equation*}
    A tuple therein is positive (or has norm at most $1$) if and only if all of its components have the same property in the original operator system (space) structures of the $S_i$.

  \item \textbf{: Equalizers.} Here things are even simpler: they are defined as for maps of vector spaces, with the operator system (space) structure inherited from that of the larger ambient space.

    The limit-preservation claim made of $U$ is now clear from these descriptions: products or equalizers are preserved because the constructions coincide, and hence arbitrary limits are preserved by \cite[Proposition 13.4]{ahs}.
  \end{enumerate}
  Given the continuity of $U$, Freyd's {\it adjoint functor theorem} (\cite[\S V.6, Theorem 2]{mcl}) reduces the adjunction claim to the {\it solution-set condition}: showing that for every operator space $X$, there is a set of completely contractive maps $\varphi_i:X\to S_i$ into operator systems $S_i$ such that every completely contractive map $\varphi:X\to S$, $S\in \cat{OSys}$, factors
  \begin{equation*}
    \begin{tikzpicture}[auto,baseline=(current  bounding  box.center)]
      \path[anchor=base] 
      (0,0) node (l) {$X$}
      +(2,.5) node (u) {$S_i$}
      +(4,0) node (r) {$S$}
      ;
      \draw[->] (l) to[bend left=6] node[pos=.5,auto] {$\scriptstyle \varphi_i$} (u);
      \draw[->] (u) to[bend left=6] node[pos=.5,auto] {$\scriptstyle \psi$} (r);
      \draw[->] (l) to[bend right=6] node[pos=.5,auto,swap] {$\scriptstyle \varphi$} (r);
    \end{tikzpicture}
  \end{equation*}
  for some $i$ and unital completely positive map $\psi$. This can be seen by simply taking for the $\varphi_i:X\to S_i$ all completely contractive maps into operator systems of (vector-space) dimension no larger than
  \begin{equation*}
    \max(\dim X,\ \aleph_0).
  \end{equation*}
  Every $\varphi:X\to S$ factors through the smallest $*$-subspace of $S$ containing $\varphi(X)$ and $e$, and that subspace has an operator system structure inherited from $S$.

  Since the forgetful functor $U$ is in fact \emph{full} (i.e. induces surjections on morphisms as well as injections) by \cite[Proposition 3.5]{pls-bk}, the \emph{reflectivity} \cite[Definition 3.5.2]{brcx_hndbk-1} claim is simply a rephrasing of $U$'s being a right adjoint.
\end{proof}

\begin{remark}\label{re:zhng.ladj}
  The existence of a left adjoint to $U$ is also effectively what \cite[Theorem 7]{MR1354984} proves, by very different means.
\end{remark}

In light of the distinction drawn in \Cref{not:osos} between categories of complete and incomplete objects, as well as the role this plays in \cite{2412.20999v1} on the \emph{local presentability} \cite[Definition 1.17]{ar} of $\cat{OSp}^{\wedge}$, it is worth calling attention to why the category of normed vector spaces (just like that of Banach spaces) is not locally \emph{finitely} presentable in the sense of \cite[Definition 1.9]{ar}. We recall some of the relevant background for the reader's benefit, denoting by $\kappa$ a \emph{regular cardinal} \cite[Definition I.10.34]{kun_st_2nd_1983}.

\begin{itemize}[wide]
  
\item A partially ordered set is \emph{$\kappa$-directed} \cite[Definition 1.13]{ar} if each subset of cardinality strictly less than $\kappa$ has an upper bound.

\item An object $c\in \cC$ of a category $\cC$ is \emph{$\kappa$-presentable} if its associated representable functor $\cC(c,-)$ preserves $\kappa$-directed colimits. The object $c$ is \emph{presentable} if $\kappa$-presentable for some regular $\kappa$.

\item $\cC$ is \emph{locally $\kappa$-presentable} \cite[Definition 1.17]{ar} if it is \emph{cocomplete} (i.e. \cite[Definition 12.2]{ahs} has colimits for all small-domain functors $\cD\to \cC$) and has a set $S$ of $\kappa$-presentable objects with every object a $\kappa$-directed colimit of objects in $S$. The category $\cC$ is \emph{locally presentable} if it is locally $\kappa$-presentable for some regular $\kappa$.
\end{itemize}

We remark that a locally $\aleph_0$-presentable category is also referred to as a \emph{locally finitely presentable} category. In addition, a locally $\aleph_1$-presentable category is called a \emph{locally countably presentable} category. Similar terminology is also used for local presentability of objects.

For our purposes, \Cref{th:osos.fgt} will provide the local $\aleph_1$-presentability noted in the following result.

\begin{corollary}\label{cor:osys.al1.pres}
  The categories $\cat{OSys}$ and $\cat{OSp}$ are locally $\aleph_1$-presentable. 
\end{corollary}
\begin{proof}
  From \Cref{th:osos.fgt}, we know that $\cat{OSys}$ is full reflective in the locally $\aleph_1$-presentable \cite[Theorem 4.35]{2412.20999v1} category $\cat{OSp}^{\wedge}$. As mentioned, that source works with \emph{complete} operator spaces; this, however, will not affect the presentability argument. From \cite[Proposition 4.30]{2412.20999v1}, one has that $\aleph_1$-directed colimits are computable set-theoretically (i.e. are preserved by the forgetful functor $\cat{OSp}\to \cat{Set}$). Hence, their completeness is automatic when the spaces involved are complete. 

  The conclusion will now follow from \cite[Theorem 1.39]{ar} together with the routine observation that an $\aleph_1$-directed colimit of operator systems in $\cat{OSp}$ inherits a natural operator-system structure making it into a colimit in $\cat{OSys}$. In short: \Cref{eq:osys2osp} is both reflective and $\aleph_1$-directed-colimit-preserving. 
\end{proof}

We note that the same strategy presented for $\cat{OSys}$ can be applied to the categories of (Archimedean) order unit spaces \cite[\S 2]{kptt}. We adopt a slightly different approach in the following observation. 

\begin{proposition}\label{pr:aou.pres}
  The locally $\aleph_1$-presentable category $\cat{AOU}$ is full reflective in the locally $\aleph_0$-presentable category $\cat{OU}$.
\end{proposition}
\begin{proof}
  The full reflectivity claim is effectively \cite[Theorem 2.38]{zbMATH05592601}. The embedding functor preserves $\aleph_1$-directed colimits, which can again be computed set-theoretically. Thus, the assertion of $\aleph_1$-presentability will follow \cite[\S 2.78, Remark]{ar} from the local $\aleph_0$-presentability of $\cat{OU}$.

  To verify the latter, we observe that
  \begin{itemize}[wide]
  \item The category $\cat{OU}$ is full reflective in the category $\cat{Vect}_{\le,1}$ of ordered vector spaces with a distinguished element $1\ge 0$.

  \item The category $\cat{Vect}_{\le,1}$ is locally $\aleph_0$-presentable because it is realizable as an \emph{essentially algebraic theory} \cite[Definition 3.34]{ar}. Indeed, this follows in a similar fashion as for vector spaces \cite[Example 3.5(2)]{ar} and partially ordered sets \cite[Example 3.35(4)]{ar}.

  \item Finally, the inclusion functor $\cat{OU}\subseteq \cat{Vect}_{\le 1}$ preserves arbitrary colimits (as does the forgetful functor to ordered vector spaces). 
  \end{itemize}
\end{proof}


On the other hand, for some of the norm-flavored categories mentioned above, the subsequent \Cref{pr:norm.not.fp} negates the aforementioned local $\aleph_0$-presentability in a strong fashion.

\begin{proposition}\label{pr:norm.not.fp}
  In any of the categories $\cat{OSp}$,  $\cat{Norm}$ or $\cat{Ban}$ the only finitely-presentable object is $\{0\}$.
\end{proposition}
\begin{proof}
  We focus on the real case to fix ideas. In all cases, there is a \emph{Hahn-Banach extension} theorem ensuring that any non-zero object $E$ admits $\bR$ as a \emph{split quotient}:
  \begin{equation*}
    \begin{tikzpicture}[>=stealth,auto,baseline=(current  bounding  box.center)]
      \path[anchor=base] 
      (0,0) node (l) {$\bR$}
      +(2,.5) node (u) {$E$}
      +(4,0) node (r) {$\bR$,}
      ;
      \draw[->] (l) to[bend left=6] node[pos=.5,auto] {$\scriptstyle \text{morphism}$} (u);
      \draw[->] (u) to[bend left=6] node[pos=.5,auto] {$\scriptstyle \text{morphism}$} (r);
      \draw[->] (l) to[bend right=6] node[pos=.5,auto,swap] {$\scriptstyle \id$} (r);
    \end{tikzpicture}
  \end{equation*}
  by \cite[Theorem 4.1.5]{er_os_2000} for $\cat{OSp}$ and \cite[Theorem III.6.2]{conw_fa} for the classical $\cat{Norm}$/$\cat{Ban}$ version. It follows from \cite[Remark following Proposition 1.16]{ar} that, if any $E\ne \{0\}$ were finitely-presentable, then $\bR$ would be as well. However, $\bR$ is not finitely-presentable. Indeed, for any of our categories $\cC$ we have
  \begin{equation*}
    \cC(\bR,-)
    \cong
    \left(\cC\xrightarrow{\quad\cat{unit ball}\quad}\cat{Set}\right),
  \end{equation*}
  so the finite presentability of $\bR$ is equivalent to the claim that the unit-ball functor preserves directed colimits. That the latter statement fails is because, for a chain of embeddings
  \begin{equation*}
    E_1\le E_2\le \cdots \le E
  \end{equation*}
  with proper dense image (dense in whatever sense appropriate to $\cC$), the colimit of $E/E_n$ is $\{0\}$ in $\cC$, yet not in $\cat{Set}$.
\end{proof}

\begin{remark}
  Cf. \cite[Proposition 4.26]{2412.20999v1}, proving the analogue for the category of \emph{complete} operator spaces. That argument does employ completeness, so does not apply directly here.
\end{remark}

\Cref{pr:norm.not.fp} (or rather its proof) has a bearing on other ways in which the categories mentioned there are not as well-behaved as they might appear. For this, we require a brief synopsis of some relevant notions:
\begin{itemize}[wide]
  
\item An \emph{epimorphism} (or \emph{epic} morphism) \cite[Definition 7.39]{ahs} $c\to c'$ in a category $\cC$ is one for which the corresponding natural transformation $\cC(c',-)\to \cC(c,-)$ is pointwise injective. 
  
\item A \emph{regular epimorphism} \cite[\S 0.5]{ar} $x\to y$ in a category is the coequalizer of a pair $\bullet\xrightarrow{f,g}x$ (as the name suggests, such morphisms are automatically epic).

\item A \emph{(regular) projective} \cite[Remark 3.4(5)]{ar} is an object $p$ with the respective lifting property
  \begin{equation}\label{eq:proj.fact}
    \begin{tikzpicture}[>=stealth,auto,baseline=(current  bounding  box.center)]
      \path[anchor=base] 
      (0,0) node (l) {$x$}
      +(2,.5) node (u) {$p$}
      +(4,0) node (r) {$y$.}
      ;
      \draw[dashed,->] (u) to[bend right=6] node[pos=.5,auto,swap] {$\scriptstyle \exists$} (l);
      \draw[->] (u) to[bend left=6] node[pos=.5,auto] {$\scriptstyle \forall$} (r);
      \draw[->>] (l) to[bend right=6] node[pos=.5,auto,swap] {$\scriptstyle \forall\text{ (regular) epic}$} (r);
    \end{tikzpicture}
  \end{equation}

\item A (small, full) subcategory $\cD\subseteq \cC$ is \emph{dense} if every object $c\in \cC$ is the canonical colimit of its $\cD$-based \emph{canonical diagram} \cite[\S 0.4]{ar}:
  \begin{equation*}
    c\cong \varinjlim_{\text{morphisms }\cD\ni d\to c}d.
  \end{equation*}
\end{itemize}

As a follow-up to \Cref{pr:norm.not.fp}, we note the dearth of regular projectives in the norm-flavored categories discussed above. One should compare our conclusion with that of \cite[Proposition 2]{MR241955}, on the absence of (non-zero) projectives in $\cat{Ban}$. 

\begin{proposition}\label{pr:no.reg.proj.norm}
  The categories $\cat{Norm}$, $\cat{Ban}$ and $\cat{OSp}$ have no non-zero regular projectives. 
\end{proposition}
\begin{proof}
  Again, we will assume that the ground field is $\bR$. In the proof of \Cref{pr:norm.not.fp}, we observed that the ground field is a split subobject of any non-zero object. The class of regular projectives being closed under split subobjects \cite[(dual to) Proposition 4.3]{ar}, $\bR$ must be one such if any are non-zero.

  A norm-1 functional $E\xrightarrow{f} \bR$ on a Banach space $E$ will produce a regular epimorphism $E\xrightarrowdbl{}E/\ker f\cong \bR$, which does not split provided $f$ does not attain its supremum on the unit ball of $E$. That this happens for \emph{some} $f$ whenever $E$ is non-reflexive \cite[Definition III.11.2]{conw_fa} is a celebrated theorem of James \cite[p.12, Theorem 3]{zbMATH03481242}.

  This type of example serves to handle both $\cat{Norm}$ and $\cat{Ban}$, while $\cat{OSp}$ follows by applying said example to the left adjoint \cite[(3.3.2) and (3.3.9)]{er_os_2000} of the forgetful functor $\cat{OSp}\to \cat{Norm}$. 
\end{proof}

\begin{remarks}\label{res:no.proj}
  \begin{enumerate}[(1),wide]
  \item\label{item:res:no.proj:semadeni.projectivity} In \cite[p.430, 2$^{nd}$ table, row 3]{sem_ban-sp-cont}, a somewhat different notion of projectivity is employed, which does allow for a broader class of such objects in $\cat{Ban}$. Therein, lifts \Cref{eq:proj.fact} are required to exist only for those horizontal $\cat{Ban}$-morphisms which restrict to surjections on unit balls. Naturally, that form of surjectivity is precisely what fails in the non-reflexive construction that \Cref{pr:no.reg.proj.norm} relies on.

  \item\label{item:res:no.proj:horn} As a consequence of \cite[Theorem 3.33(i) and Remark 5.13(3)]{ar}, the locally $\aleph_1$-presentable categories of \Cref{pr:no.reg.proj.norm} cannot be axiomatized as the \emph{(infinitary) quasi-varieties} of \cite[\S 3.C]{ar} or the \emph{universal Horn theories} of \cite[\S 5.A and Example 5.27(3)]{ar}. 
  \end{enumerate}  
\end{remarks}

\subsection{Convexity and function systems}\label{subse:fn.sys}

Within this subsection, we consider a chain of subcategories
\begin{equation}\label{eq:smplx.chn}
  \cat{OSys}_{bs}^{\wedge}
  \subset
  \cat{OSys}_{s}^{\wedge}
  \subset
  \cat{OSys}_{c}^{\wedge}
  \subset
  \cat{OSys}^{\wedge}
  \subset
  \cat{OSys},
\end{equation}
whose notation we shall now elaborate upon.

\begin{itemize}[wide]
\item The `c' subscript indicates the subcategory consisting of (complete) operator systems that may be embedded within a commutative $C^*$-algebra. In other words, those $S\in \cat{OSys}^{\wedge}$ for which there is a compact Hausdorff space $X$ and a unital complete order isomorphism $S\hookrightarrow\mathrm{C}(X)$.

  A theorem of Kadison \cite[Theorem II.1.8]{alfsen2012compact} asserts that $\cat{OSys}^{\wedge}_{c}$ is contravariantly equivalent to the category of compact convex subsets of locally convex topological vector spaces, with continuous affine maps as their morphisms. The equivalence is implemented by considering the space of continuous affine functions $\mathrm{A}(K)$ for a compact convex set $K$. The corresponding inverse is the compact convex set $S^*_+$ given by positivity-preserving functionals for an object $S$ in $\cat{OSys}^{\wedge}_{c}$.

\item The embedding $\cat{OSys}_{s}^{\wedge}\subset
  \cat{OSys}_{c}^{\wedge}$ in \Cref{eq:smplx.chn} is the embedding dual to
  \begin{equation*}
    \text{\emph{(Choquet) simplices} \cite[\S II.3, p.84]{alfsen2012compact}}
    \subset
    \text{compact convex spaces}.
  \end{equation*}

\item The $\cat{OSys}_{bs}^{\wedge}\subset\cat{OSys}_{s}^{\wedge}$ embedding is that dual to
  \begin{equation*}
    \text{\emph{Bauer simplices} \cite[Theorem II.4.1]{alfsen2012compact}}
    \subset
    \text{simplices},
  \end{equation*}
  with the `Bauer' modifier referring (following \cite{zbMATH03159688}) to those simplices whose sets of extreme points are closed. We also refer to \cite[\S II.4]{alfsen2012compact} and \cite[\S 5, post Example 5.8]{2412.09455v1} for background information on these notions. The Bauer simplices are precisely the compact convex sets of the form
  \begin{equation*}
    \mathrm{Prob}(X)
    :=
    \left\{\text{probability measures on $X$}\right\}
    ,\quad
    \text{compact Hausdorff $X$}. 
  \end{equation*}  
\end{itemize}Consequently, the Bauer simplices are those compact convex sets for which the space of affine functions $\mathrm{A}(K)$ is completely order isomorphic to a commutative $C^*$-algebra.


\begin{theorem}\label{th:osfcn.fgt}
The following statements hold.
  \begin{enumerate}[(1),wide]
  \item\label{item:th:osfcn.fgt:c} The category $\cat{OSys}^{\wedge}_{c}$ is complete and the inclusion functor $U_c: \cat{OSys}^{\wedge}_{c}\rightarrow \cat{OSys}$ is a right adjoint, so that $\cat{OSys}^{\wedge}_{c}$ is full reflective.

  \item\label{item:th:osfcn.fgt:s} The category $\cat{OSys}^{\wedge}_{s}$ has products and the inclusion functor $U_s: \cat{OSys}^{\wedge}_{s}\rightarrow \cat{OSys}^{\wedge}_{c}$ preserves them.
  \end{enumerate}  
\end{theorem}
\begin{proof}
  We first describe products and equalizers. To this end, we use a characterization of simplices in terms of their corresponding space of affine functions \cite[Corollary II.3.11]{alfsen2012compact}. That is, each object in $\cat{OSys}^{\wedge}_{s}$ can be identified as an operator subsystem $S\subset\mathrm{C}(X)$ of continuous functions on a compact Hausdorff space $X$ with the so-called Riesz interpolation property.
  
  \begin{enumerate}[(I), wide]
  \item \textbf{: Products.} Checking that $\cat{OSys}^{\wedge}_{c}$ has products is near immediate. For $\cat{OSys}^{\wedge}_{s}$, suppose that $\{S_i : i\in I\}$ is a family of objects in $\cat{OSys}^{\wedge}_{s}$. Now, it suffices to check that $\prod_i S_i$, which is an object in $\cat{OSys}_{c}$, satisfies the Riesz interpolation property \cite[Corollary II.3.11]{alfsen2012compact}. To this end, assume that $x_1,x_2\leq y_1, y_2$ in $\prod_i S_i$. As each $S_i$ has the Riesz interpolation property, it follows that there are $\{z_i\in S_i : i\in I\}$ such that \[x_1(i),x_2(i) \leq z_i\leq y_1(i), y_2(i), \ \ \ \  i\in I.\]Choosing $z\in\prod_i S_i$ defined by $z(i):=z_i$ guarantees that \[x_1, x_2\leq z\leq y_1, y_2.\] Therefore, $\prod_i S_i$ has the Riesz interpolation property, and belongs to $\cat{OSys}^{\wedge}_{s}$ by \cite[Corollary II.3.11]{alfsen2012compact}.
    
  \item \textbf{: Equalizers.} The objects in $\cat{OSys}_c^\wedge$ are defined as for maps of vector spaces, and the operator system structure is inherited from that of an ambient commutative $C^*$-algebra.
  \end{enumerate}
  
  At this point, the argument follows through as in \Cref{th:osos.fgt}.
\end{proof}

\begin{remark}\label{re:kad.adj}
  The existence of a left adjoint to $U_c$ is a component of Kadison's duality theorem \cite[Theorem II.1.8.]{alfsen2012compact}, albeit it is accomplished through analytic machinery instead.
\end{remark}

We now note that the analogue of \Cref{th:osfcn.fgt}\Cref{item:th:osfcn.fgt:c} cannot hold for the inclusions
\begin{equation}\label{eq:smplx.incl}
  \cat{OSys}^{\wedge}_{bs}
  \quad\text{or}\quad
  \cat{OSys}^{\wedge}_{s}
  \quad
  \lhook\joinrel\xrightarrow{\quad}
  \quad
  \cat{OSys}^{\wedge}_{c}
  ,\ 
  \cat{OSys}
  ,\ \text{etc.}
\end{equation}

\begin{proposition}\label{pr:smplx.no.eqs}
  The categories $\cat{OSys}^{\wedge}_{\bullet}$, $\bullet\in \{s,bs\}$ do not have equalizers. 
\end{proposition}
\begin{proof}
  We dualize via \cite[Theorem II.1.8]{alfsen2012compact}, working directly with Bauer/Choquet simplices. Denote the $n$-dimensional simplex by $\Delta^n$, and consider two maps $\Delta^0\xrightarrow{f,g} \Delta^3$ sending the singleton to the two midpoints of two opposite edges in a tetrahedron. The coequalizer of the diagram in the category of compact convex spaces is a square:
  \begin{equation*}
    I^2
    \subset \bR^2
    ,\quad
    I:=[0,1]
    ,\quad
    \text{usual convex structure}.
  \end{equation*}
  The latter cannot map \emph{universally} \cite[\S III.1]{mcl_2e} into any (Choquet or Bauer) simplex. Assuming this settled, the conclusion that $f$ and $g$ cannot have a coequalizer in either the category of Bauer or that of Choquet simplices follows; it does remains to verify the claimed non-universality. 

  Assuming otherwise, consider such a universal morphism $I^2\to \Delta$ to a Choquet simplex; it must be an embedding, for $I^2$ embeds into a finite-dimensional simplex; we thus identify $I^2$ with a compact convex subset of $\Delta$. 

  For every three of the four extreme points of $I^2$ there is some compact-convex-space morphism $I^2\to \Delta^2$ mapping said three vertices to the three extreme points of $\Delta^2$. It follows that the \emph{closed faces} \cite[\S 1.1, pre Theorem 1.3]{ae_cvx} $F_i$, $0\le i\le 3$ respectively generated by the four vertices $p_i$ of $I^2$ are mutually disjoint. The relation
  \begin{equation*}
    \frac{p_0+p_2}2 = \frac{p_1+p_3}2
  \end{equation*}
  yields probability measures supported on $F_0\cup F_2$ and $F_1\cup F_3$ with the same \emph{resultant} \cite[p.1, Definition]{phlps_choq}; this is an \emph{affine dependence} \cite[(1.5)]{zbMATH03302675}, violating (one formulation of) the simplex property. 
\end{proof}

As an immediate consequence to \Cref{pr:smplx.no.eqs}, we can also state the following of the inclusions in \Cref{eq:smplx.incl}.

\begin{corollary}\label{cor:smplx.not.refl}
  The full inclusion functors of Equation \Cref{eq:smplx.incl} are not reflective. 
\end{corollary}
\begin{proof}
  The codomains are complete, and it is a routine exercise to verify that full reflective subcategories of complete categories are closed under limits (and, in particular, remain complete). This is certainly well-known, and sketched in \cite[Exercise 3.F]{freyd_abcats}, say. Denote such a reflection by
  \begin{equation*}
    \begin{tikzpicture}[>=stealth,auto,baseline=(current  bounding  box.center)]
      \path[anchor=base] 
      (0,0) node (l) {$\cC$}
      +(4,0) node (r) {$\cC'$}
      +(2,0) node () {$\bot$}
      ;
      \draw[->] (l) to[bend left=20] node[pos=.5,auto] {$\scriptstyle F:=\cat{reflection}$} (r);
      \draw[right hook->] (r) to[bend left=20] node[pos=.5,auto] {$\scriptstyle \iota$} (l);
    \end{tikzpicture}
  \end{equation*}
  and note that an object $c\in \cC$ belongs to $\cC'\simeq \iota \cD$ precisely when the unit $c\to \iota Fc$ is an isomorphism. This condition holds for the limit
  \begin{equation*}
    \varprojlim \iota T\in \cC
    ,\quad
    \cD\xrightarrow[\quad\text{functor}\quad]{\quad T\quad}\cC',
  \end{equation*}
  which comes equipped with a natural candidate
  \begin{equation*}
    \iota F\left(\varprojlim \iota T\in \cC\right)
    \xrightarrow{\quad}
    \varprojlim \iota T\in \cC
  \end{equation*}
  induced by the universality property of that limit. The conclusion follows from the failure of completeness concluded in \Cref{pr:smplx.no.eqs}. 
\end{proof}




With regard to presentability, the function-system variation to \Cref{cor:osys.al1.pres} becomes a consequence of \Cref{th:osfcn.fgt}.

\begin{theorem}\label{th:osfcn.a1.pres}
  The category $\cat{OSys}^{\wedge}_c$ is locally $\aleph_1$-presentable. 
\end{theorem}
\begin{proof}
  The proof proceeds precisely along the lines of \Cref{cor:osys.al1.pres}, employing \cite[Theorem 1.39]{ar}: it suffices to argue that the inclusion functor $U_c$ of \Cref{th:osfcn.fgt} preserves $\aleph_1$-directed colimits. We have repeatedly observed that these are computed set-theoretically in the larger category $\cat{OSys}$, so we need to verify this in the smaller category $\cat{OSys}^{\wedge}_s$.

  An $\aleph_1$-directed system $S_{\lambda'}\xleftarrow{\psi_{\lambda',\lambda}}S_{\lambda}$ in $\cat{OSys}^{\wedge}_c$ corresponds via Kadison's representation theorem \cite[Theorem II.1.8]{alfsen2012compact} to a system
  \begin{equation*}
    K_{\lambda}
    \xleftarrow{\quad\varphi_{\lambda,\lambda'}\quad}
    K_{\lambda'}
    ,\quad
    \lambda\le \lambda'
  \end{equation*}
  of compact convex sets (with affine continuous functions). The claim, in context, reduces to the affine functions on $K:=\varprojlim K_{\lambda}$ being precisely those continuous functions that factor as
  \begin{equation*}
    K
    \xrightarrow{\quad\text{canonical map}\quad}
    K_{\lambda}
    \xrightarrow[\quad\text{continuous affine}\quad]{\quad}
    \bC,
  \end{equation*}
  with two such functions being equal when restricting to equal functions on some $K_{\lambda'}$. This, in turn, follows from the analogous statements in the category $C^*_{c}$ of (unital) commutative $C^*$-algebras: that $\aleph_1$-directed colimits are computed set-theoretically there is checked as in \cite[Proposition 4.30]{2412.20999v1}, and the functions being affine carries through. Indeed, recall that a continuous function $f\in C(K)$ is affine if and only if
  \begin{equation}\label{eq:def.aff}
    \forall\left(\lambda\in [0,1]\right)
    \forall\left(x,y\in K\right)
    \quad:\quad
    f(\lambda x+(1-\lambda)y)=\lambda f(x)+(1-\lambda)f(y). 
  \end{equation}
  Expressing this as an equality $f_{in}(\lambda,x,y)=f_{out}(\lambda,x,y)$
  with $f_{in},f_{out}\in C([0,1],K^2)$ defined as the left-hand and right-hand sides of the equality in \Cref{eq:def.aff}, respectively. Observe next that
  \begin{equation}\label{eq:int.k2}
    [0,1]\times K^2\cong \varprojlim_{\lambda}\left([0,1]\times K_{\lambda}^2\right),
  \end{equation}
  and thus, the desired isomorphism
  \begin{equation*}
    A(K)\cong \varinjlim_{\lambda}A(K_{\lambda})
  \end{equation*}
  follows from its analogue for continuous functions (i.e. $C(\bullet)$ in place of $A(\bullet)$) applied to both $K\cong \varprojlim_{\lambda}K_{\lambda}$ and \Cref{eq:int.k2}. 
\end{proof}

Finally, we consider operator space counterparts $\cat{OSp}^{\bullet}_c$ with $\bullet\in \left\{\text{blank},\wedge\right\}$ and `$\wedge$' denoting the complete objects, as before. The category $\cat{OSp}_c$ consists of the operator spaces that can be completely isometrically embedded into a commutative $C^*$-algebra. In \cite[Proposition 3.3.1]{er_os_2000}, we see that $\cat{OSp}_c$ are identified as those operator spaces that are \emph{minimal} or \emph{commutative} \cite[\S 3.3, discussion post (3.3.11)]{er_os_2000}.

Writing $\cat{Norm}^{\wedge}:=\cat{Ban}$ for the sake of notation uniformity, there are adjunctions \cite[\S 3.3]{er_os_2000}
\begin{equation}\label{eq:osp.min}
  \begin{tikzpicture}[>=stealth,auto,baseline=(current  bounding  box.center)]
    \path[anchor=base] 
    (0,0) node (l) {$\cat{OSp}^{\bullet}$}
    +(4,0) node (r) {$\cat{Norm}^{\bullet}$}
    +(2,0) node () {$\bot$}
    ;
    \draw[->] (l) to[bend left=20] node[pos=.5,auto] {$\scriptstyle U:=\cat{forget}$} (r);
    \draw[->] (r) to[bend left=20] node[pos=.5,auto] {$\scriptstyle \cat{min}$} (l);
  \end{tikzpicture}
\end{equation}
with the tail of the `$\bot$' symbol pointing toward the left adjoint, as is customary \cite[Definition 19.3]{ahs}. We remark that we are committing a slight abuse in notating the two adjunctions by identical symbols. The \emph{minimal} operator spaces (complete or not) are those for which the canonical map $S\to \left(\cat{min}\circ\cat U\right) S$ is an isomorphism. This yields the following operator space variant to \Cref{th:osfcn.a1.pres}. 

\begin{theorem}\label{th:comm.osp}
  The categories $\cat{OSp}^{\bullet}_c$, $\bullet\in \left\{\text{blank},\wedge\right\}$ are locally $\aleph_1$-presentable. 
\end{theorem}
\begin{proof}
  The fact that $\cat{min}$ is a \emph{strict quantization} in the sense of \cite[\S 3.3]{er_os_2000} means in particular that the counit $U\circ\cat{min}\to \id$ of \Cref{eq:osp.min} is an isomorphism. This gives equivalences $\cat{OSp}^{\bullet}_c\simeq \cat{Norm}^{\bullet}$, and the conclusion follows from the latter's local $\aleph_1$-presentability (\cite[Example 1.48]{ar}, \cite[Proposition 2.1]{zbMATH07469564}, etc.).
\end{proof}

\section{Monadicity for non-self-adjoint operator algebras}\label{se:osoa}

We denote by $\cat{OAlg}$ the category of unital \emph{abstract unital operator algebras} \cite[\S 16, pp.232-233]{pls-bk}: operator spaces with a unital algebra structure such that multiplication is \emph{completely contractive} in the sense of \cite[paragraph post Corollary 16.5]{pls-bk} (\emph{multiplicatively} contractive in the language of \cite[\S 9.1]{er_os_2000}).

The \emph{Haagerup tensor product} $\otimes_h$ \cite[\S 16, p.240]{pls-bk} (or \cite[\S 9.2]{er_os_2000}) makes $\cat{OSp}$ into a monoidal category $(\cat{OSp},\ \otimes_h,\ \mathbf{1}=\bC)$ (\cite[post Proposition 3.1]{zbMATH04050295} or \cite[Proposition 9.2.7]{er_os_2000} deliver associativity). The very definition of an abstract operator algebra and the universality property \cite[Proposition 9.2.2]{er_os_2000} of $\otimes_h$ give the identification
\begin{equation*}
  \cat{OAlg}
  \simeq
  \cat{Mon}\left(\cat{OSp},\ \otimes_h,\ \mathbf{1}\right)
\end{equation*}
with the category of \emph{monoids} \cite[\S 2.2]{zbMATH05312006} (or internal unital, associative algebras) attached to the monoidal structure.

The succeeding \Cref{th:env.oalg} will be an analogue of \Cref{th:osos.fgt} for $\cat{OAlg}$. We recall some terminology. 
\begin{itemize}[wide]
\item A \emph{monad} \cite[Definition 20.1]{ahs} on a category $\cC$ is an endofunctor $\cC\xrightarrow{T}\cC$ equipped with natural transformations
  \begin{equation*}
    T\circ T\xrightarrow{\quad\mu\quad}T
    \quad\text{and}\quad
    \id\xrightarrow{\quad\eta\quad}T
  \end{equation*}
  making $T$ into a unital associative algebras in the monoidal category of $\cC$-endofunctors with composition as the tensor product.

\item For a monad $(T,\mu,\eta)$, the (\emph{Eilenberg-Moore}) category $\cC^T$ of \emph{$T$-algebras} \cite[Definition 20.4]{ahs} consists of objects $c\in \cC$ equipped with morphisms $Tc\to c$ associative and unital with respect to $\mu$ and $\eta$ in the sense familiar from defining modules over algebras.

\item A functor $\cD\xrightarrow{U}\cC$ is \emph{monadic} if there is a monad $T$ on $\cC$ so that
  \begin{equation*}
    \begin{tikzpicture}[>=stealth,auto,baseline=(current  bounding  box.center)]
      \path[anchor=base] 
      (0,0) node (l) {$\cD$}
      +(2,.5) node (u) {$\cC^T$}
      +(4,0) node (r) {$\cC$}
      ;
      \draw[->] (l) to[bend left=6] node[pos=.5,auto] {$\scriptstyle \simeq$}
      (u);
      \draw[->] (u) to[bend left=6] node[pos=.5,auto] {$\scriptstyle \cat{forget}$} (r);
      \draw[->] (l) to[bend right=6] node[pos=.5,auto,swap] {$\scriptstyle U$} (r);
    \end{tikzpicture}
  \end{equation*}
  commutes up to natural isomorphism.
  
  Monads are characterized by \emph{Beck's theorem} (\cite[Theorem 20.17]{ahs}, \cite[\S 3.3, Theorem 10]{bw}) and, in particular, they are right adjoints. Fully faithful right adjoints are monadic attached to an \emph{idempotent} monad \cite[Proposition 4.2.3]{brcx_hndbk-2}, and so \Cref{th:osos.fgt} is effectively a monadicity claim (much like \Cref{th:env.oalg} below).
\end{itemize}

\begin{theorem}\label{th:env.oalg}
  The categories $\cat{OAlg}$ and $\cat{OSp}$ are complete, and the forgetful functor $\cat{OAlg}\xrightarrow{U}\cat{OSp}$ is monadic. 
\end{theorem}
\begin{proof}
  Much of the proof parallels that of \Cref{th:osos.fgt}.

  \begin{enumerate}[(I),wide]
  \item \textbf{: Completeness.} The $\cat{OSp}$-products (given by \cite[Proposition 4.11]{2412.20999v1}) and equalizers \cite[Proposition 4.13]{2412.20999v1} plainly function for $\cat{OAlg}$ just as well as they do for operator spaces, and $U$ preserves them. It then follows \cite[Proposition 13.4]{ahs} that $U$ is continuous.

  \item\label{item:th:env.oalg:pf.adj} \textbf{: Adjointness.} To conclude via \cite[Theorem 18.12]{ahs}, it again suffices to check the solution-set condition: given an operator space $E$, there is a set of complete contractions $E\xrightarrow{\varphi_s} A_s$ into operator algebras such that every complete contraction $E\to A$ factors as
    \begin{equation*}
      \begin{tikzpicture}[>=stealth,auto,baseline=(current  bounding  box.center)]
        \path[anchor=base] 
        (0,0) node (l) {$E$}
        +(2,.5) node (u) {$A_s$}
        +(4,0) node (r) {$A$}
        ;
        \draw[->] (l) to[bend left=6] node[pos=.5,auto] {$\scriptstyle \varphi_s$} (u);
        \draw[dashed,->] (u) to[bend left=6] node[pos=.5,auto] {$\scriptstyle $} (r);
        \draw[->] (l) to[bend right=6] node[pos=.5,auto,swap] {$\scriptstyle \varphi$} (r);
      \end{tikzpicture}
    \end{equation*}
    Simply take for the $\varphi_s$ those morphisms whose codomain $A_s$ is generated as an algebra by a subspace of dimension $\le \dim E$. 

  \item\label{item:th:env.oalg:pf.mndc} \textbf{: Monadicity.} We will apply the aforementioned \cite[\S 3.3, Theorem 10]{bw}. The first of that result's three hypotheses ($U$'s being a right adjoint) has been verified above. The second requirement is that $U$ \emph{reflect isomorphisms} \cite[\S 3.3, p.103]{bw}:
    \begin{equation*}
      \forall \left(A\xrightarrow{\quad f\quad}B\right)\in \cat{OAlg}
      \quad:\quad
      Uf\text{ invertible in }\cat{OSp}
      \xRightarrow{\quad}
      f\text{ invertible};
    \end{equation*}
    this is immediate, the isomorphisms in $\cat{OAlg}$ being precisely the unital, multiplicative, completely isometric bijections. It thus remains to check that whenever a pair $A\xrightarrow{d_i,\ i=0,1}B$ of morphisms in $\cat{OAlg}$
    \begin{itemize}[wide]
    \item is \emph{reflexive} \cite[\S 3.3, p.108]{bw} in the sense that the two arrows have a common right inverse;

    \item and \emph{$U$-contractible} \cite[\S 3.3, p.105]{bw} in the sense of affording a diagram
      \begin{equation}\label{eq:cntrct.coeq}
        \begin{tikzpicture}[>=stealth,auto,baseline=(current  bounding  box.center)]
          \path[anchor=base] 
          (0,0) node (l) {$UA$}
          +(3,0) node (m) {$UB$}
          +(6,0) node (r) {$E$}
          ;
          \draw[->] (l) to[bend left=30] node[pos=.5,auto] {$\scriptstyle Ud_0$} (m);
          \draw[<-] (l) to[bend left=0] node[pos=.5,auto] {$\scriptstyle t$} (m);
          \draw[->] (l) to[bend right=30] node[pos=.5,auto,swap] {$\scriptstyle Ud_1$} (m);
          \draw[->] (m) to[bend left=10] node[pos=.5,auto] {$\scriptstyle d$} (r);
          \draw[->] (r) to[bend left=10] node[pos=.5,auto] {$\scriptstyle s$} (m);
        \end{tikzpicture}
      \end{equation}
      in $\cat{OSp}$ with
      \begin{equation*}
        ds=\id
        ,\quad
        (Ud_0) t=\id
        \quad\text{and}\quad
        (Ud_1)t=sd,
      \end{equation*}      
    \end{itemize}
    that pair has a coequalizer preserved by $U$. The monadicity of the forgetful functor from unital associative algebras to vector spaces (e.g. \cite[\S 2.6, Corollary]{zbMATH05312006}) already ensures that under the circumstances the range
    \begin{equation}\label{eq:id0d1}
      I
      :=
      \im(d_0-d_1)
      =
      \left\{(d_0-d_1)(a)\ :\ a\in A\right\}
      \le
      B
    \end{equation}
    is an ideal, so the \emph{operator space quotient} \cite[Proposition 3.1.1]{er_os_2000} $\overline{B}:=B/I$ inherits a unital algebra structure from $B$. That this gives $\overline{B}$ an operator-algebra structure (i.e. its multiplication is continuous for the Haagerup tensor norm) follows from the factorization
    \begin{equation*}
      \begin{tikzpicture}[>=stealth,auto,baseline=(current  bounding  box.center)]
        \path[anchor=base] 
        (0,0) node (l) {$B\otimes_h B$}
        +(4,.5) node (u) {$B$}
        +(2,-1) node (dl) {$B\otimes_h B/\left(I\otimes B+B\otimes I\right)$}
        +(6,-1) node (dr) {$\overline{B}\otimes_h \overline{B}$}
        +(8,0) node (r) {$\overline{B}$,}
        ;
        \draw[->] (l) to[bend left=6] node[pos=.5,auto] {$\scriptstyle \text{completely contractive}$} node[pos=.5,auto,swap] {$\scriptstyle \text{mult.}$} (u);
        \draw[->>] (u) to[bend left=6] node[pos=.5,auto] {$\scriptstyle $} (r);
        \draw[->] (l) to[bend right=6] node[pos=.5,auto,swap] {$\scriptstyle $} (dl);
        \draw[->] (dl) to[bend right=6] node[pos=.5,auto,swap] {$\scriptstyle \cong$} (dr);
        \draw[->] (dr) to[bend right=6] node[pos=.5,auto,swap] {$\scriptstyle $} (r);
      \end{tikzpicture}
    \end{equation*}
    the bottom map being an isomorphism of operator spaces by the \emph{projectivity} \cite[Corollary 5.7(ii)]{pis_os} of the Haagerup tensor product. 
  \end{enumerate}
\end{proof}

\begin{remark}\label{re:pisier.osp2oalg}
  In reference to step \Cref{item:th:env.oalg:pf.adj} in the preceding proof, see also the explicit construction of the left adjoint to $\cat{OAlg}^\wedge\xrightarrow{U}\cat{OSp}^\wedge$ in \cite[pre Remark 6.5]{pis_os}.
\end{remark}

Before proceeding, we record a consequence of the proof of \Cref{th:env.oalg}.

\begin{corollary}\label{cor:cast2osys.mndc}
  The forgetful functor $\cat{C}^*\to \cat{OSys}$ from unital $C^*$-algebras to operator systems is monadic. 
\end{corollary}
\begin{proof}
  Step \Cref{item:th:env.oalg:pf.mndc} in the proof of \Cref{th:env.oalg} goes through with nothing but the obvious modifications. One point perhaps worth making explicitly is that because the ideal $I$ of \Cref{eq:id0d1} is also the kernel of the continuous idempotent map $sd=d_1 t$ of \Cref{eq:cntrct.coeq}, it is in the present context automatically closed and hence $B/I$ acquires a $C^*$-algebra structure.
\end{proof}

The monadicity of \Cref{th:env.oalg} will now allow us to lift the local presentability from $\cat{OSp}$ to $\cat{OAlg}$.

\begin{theorem}\label{th:oalg.al1}
  The monad $\cat{OSp}\xrightarrow{T}\cat{OSp}$ attached to the forgetful functor $U$ of \Cref{th:env.oalg} preserves $\aleph_1$-directed colimits. In particular, $\cat{OAlg}\simeq \cat{OSp}^T$ is locally $\aleph_1$-presentable. 
\end{theorem}
\begin{proof}
  The second claim follows from the first by \cite[\S 2.78, Remark]{ar}, given the local $\aleph_1$-presentability of $\cat{OSp}$ \cite[Theorem 4.35]{2412.20999v1}.

  It remains to argue that $T$ preserves $\aleph_1$-directed colimits; we do so by arguing that $U$ does (also verifying in passing that those colimits do exist in $\cat{OAlg}$). To this end, consider a $\aleph_1$-directed colimit
  \begin{equation*}
    A_{\alpha}
    \xrightarrow{\quad\psi_{\alpha}\quad}
    A
    =
    \varinjlim_{\lambda\in \left(\Lambda,\le\right)}A_{\lambda}
  \end{equation*}
  in $\cat{OSp}$, whose objects and connecting maps $A_{\lambda'}\xleftarrow{\psi_{\lambda',\lambda}}A_{\lambda}$ are, in fact, in $\cat{OAlg}$.

  As a set $A$ is nothing but the colimit of the corresponding $\cat{Set}$ diagram (so the $\psi_{\lambda}$ are unital-algebra morphisms), and the matricial norms \cite[Proposition 4.30]{2412.20999v1} on $M_n(A)$ are the expected
  \begin{equation*}
    \|a\|=\inf\left\{\left\|a_{\lambda}\right\|\text{ in }M_n(A_{\lambda})\ :\ \psi_{\lambda}(a_{\lambda})=a\right\}
    ,\quad
    \forall
    a\in M_n(A). 
  \end{equation*}
  It remains to argue that these matricial norms are compatible with the algebra structure, in the sense that the multiplication is continuous on the Haagerup tensor square $A\otimes_h A$. This in turn follows from the fact that the Haagerup tensor product preserves $\aleph_1$-directed colimits, relegated to \Cref{le:haag.al1.pres}. 
\end{proof}

We follow a number of standard sources (e.g. \cite[(9.1.7)]{er_os_2000} or \cite[Chapter 5]{pis_os}) in denoting by
\begin{equation*}
  M_{mk}(E)\otimes M_{kn}(F)
  \ni
  (e_{ij})\otimes (f_{jk})
  \xmapsto{\quad\odot\quad}
  \left(\sum_j e_{ij}\otimes f_{jk}\right)_{i,k}
  \in
  M_{mn}(E\otimes F)
\end{equation*}
the bilinear maps ubiquitous in work on Haagerup tensor products.

\begin{lemma}\label{le:haag.al1.pres}
  For $\aleph_1$-directed partially ordered sets $(\Lambda_{\bullet},\ \le)$, $\bullet\in\{0,1\}$ and $\Lambda_{\bullet}$-indexed diagrams $(E_{\lambda})_{\lambda\in \Lambda_{\bullet}}$ in $\cat{OSp}$ the canonical map
  \begin{equation*}
    \varinjlim_{\lambda_{\bullet}}E_{\lambda_0}\otimes_h F_{\lambda_1}
    \xrightarrow{\quad}
    \left(\varinjlim_{\lambda_0}E_{\lambda_0}\right)
    \otimes_h
    \left(\varinjlim_{\lambda_1}F_{\lambda_1}\right)
  \end{equation*}
  is an isomorphism in $\cat{OSp}$. 
\end{lemma}
\begin{proof}
  Given colimit-colimit commutation \cite[\S IX.2, equation (2)]{mcl}, the claim reduces to showing that $E\otimes_h-$ and $-\otimes_h F$ both preserve $\aleph_1$-directed colimits. We focus on the former, to fix ideas, seeking to show that the canonical bijection, which is completely contractive,
  \begin{equation}\label{eq:can.colim.map}
    \varinjlim_{\lambda\in \left(\Lambda,\le\right)}\left(E\otimes_h F_{\lambda}\right)
    \lhook\joinrel\xrightarrowdbl{\qquad}
    E\otimes_h\left(\varinjlim_{\lambda} F_{\lambda}\right)
    \quad
    \left(\text{$\aleph_1$-directed}\right),
  \end{equation}
  is in fact completely isometric. The argument will make it clear that the right-handed counterpart is entirely parallel (the asymmetry \cite[Propositions 9.3.1 and 9.3.2]{er_os_2000} of $\otimes_h$ notwithstanding).  

  In general, for a directed set colimit $S=\varinjlim_{\lambda}S_{\lambda}$ of sets, we refer to an element $s_{\lambda}\in S_{\lambda}$ mapping onto $s\in S$ as a \emph{preimage} (or \emph{$\lambda$-preimage}) of $s$. Consider an element
  \begin{equation*}
    x\in M_{mn}\left(E\otimes_h F\right)
    ,\quad
    F:=
    \left(\varinjlim_{\lambda} F_{\lambda}\right).
  \end{equation*}
  The Haagerup norm $\|x\|_h$ of $x$ is defined by
  \begin{equation}\label{eq:xhnorm}
    \inf\left\{
      \left\|(e_{ij})\right\|\cdot \left\|(f_{jk})\right\|
      \ :\
      (e_{ij})\in M_{mk}(E)
      ,\
      (f_{ij})\in M_{kn}(F)
      \ \text{and}\ 
      (e_{ij})\odot (f_{jk}) = x
    \right\}.
  \end{equation}
  We can select $(e^{(d)}_{ij})$ and $(f^{(d)}_{jk})$ as in \Cref{eq:xhnorm} with
  \begin{equation*}
    \left\|(e^{(d)}_{ij})\right\|\cdot \left\|(f^{(d)}_{jk})\right\|\le \|x\|_h+\frac 1d
    ,\quad
    d\in \bZ_{>0}.
  \end{equation*}
    The fact $\Lambda$ is $\aleph_1$-directed guarantees that the preimages $f^{(d)}_{\lambda\mid jk}\in F_{\lambda}$ of the $f^{(d)}_{jk}$ are in $M_{kn}\left(F_{\lambda}\right)$ for sufficiently large $\lambda$. Thus, their respective norms agree with those of $f^{(d)}_{jk}$. This then provides preimages
  \begin{equation*}
    x_{\lambda} := \left(e^{(d)}_{ij}\right)\odot \left(f^{(d)}_{\lambda\mid jk}\right)
    \in
    M_{mn}\left(E\otimes F_{\lambda}\right)
  \end{equation*}
  of $x$ with Haagerup norm $\frac 1d$-close to that of $x$. By enlarging $\lambda$ if necessary (again recalling that $\Lambda$ is $\aleph_1$-directed), we have preimages of $x$ in $M_{mn}\left(E\otimes F_{\lambda}\right)$ whose Haagerup norm is tends to $\|x\|_h$. 

  A third application of $\Lambda$ being $\aleph_1$-directed allows us to assume that the norms of $\iota_{\lambda',\lambda}\left(x_{\lambda}\right)$ are all equal for the connecting maps
  \begin{equation*}
    E\otimes_h F_{\lambda}
    \xrightarrow{\quad \iota_{\lambda',\lambda}\quad}
    E\otimes_h F_{\lambda'}
  \end{equation*}
  effecting the left-hand colimit in \Cref{eq:can.colim.map}, and the conclusion follows from our having transported $x$ to the bottom corner of the commutative diagram
  \begin{equation*}
    \begin{tikzpicture}[>=stealth,auto,baseline=(current  bounding  box.center)]
      \path[anchor=base] 
      (0,0) node (l) {$\varinjlim_{\lambda\in \left(\Lambda,\le\right)}\left(E\otimes_h F_{\lambda}\right)$}
      +(3,-.5) node (d) {$E\otimes F_{\lambda}$}
      +(6,0) node (r) {$E\otimes_h\left(\varinjlim_{\lambda} F_{\lambda}\right)$}
      ;
      \draw[right hook->>] (l) to[bend left=6] node[pos=.5,auto] {$\scriptstyle $} (r);
      \draw[->] (d) to[bend left=6] node[pos=.5,auto] {$\scriptstyle $} (l);
      \draw[->] (d) to[bend right=6] node[pos=.5,auto] {$\scriptstyle $} (r);      
    \end{tikzpicture}
  \end{equation*}
  with the bottom maps preserving its norm. 
\end{proof}

\begin{remark}\label{re:met.al1.pres}
  Ultimately, the crux of the proof of \Cref{le:haag.al1.pres} is the $\kappa$-presentability, for every uncountable regular cardinal $\kappa$, of metric spaces of cardinality $\kappa$ in the locally $\aleph_1$-presentable category $\cat{Met}$ of metric spaces with contractions \cite[Example 2.3(1)]{zbMATH07461229}. This follows immediately from \cite[Lemma 2.4 supplemented by Remark 2.5(1)]{zbMATH07461229}, and is also ultimately what drives both \cite[Lemma 3.4 and Proposition 3.5]{2412.20999v1}.
\end{remark}

All in all, the functors discussed in \Cref{th:osos.fgt,th:env.oalg} and \Cref{cor:cast2osys.mndc} fit into the following diagram of monadic forgetful/inclusion functors between locally $\aleph_1$-presentable categories:
\begin{equation*}
  \begin{tikzpicture}[>=stealth,auto,baseline=(current  bounding  box.center)]
    \path[anchor=base] 
    (0,0) node (l) {$\cat{C}^*$}
    +(2,.5) node (u) {$\cat{OAlg}$}
    +(2,-.5) node (d) {$\cat{OSys}$}
    +(4,0) node (r) {$\cat{OSp}$}
    ;
    \draw[right hook->] (l) to[bend left=6] node[pos=.5,auto] {$\scriptstyle $} (u);
    \draw[->] (u) to[bend left=6] node[pos=.5,auto] {$\scriptstyle \cat{forget}$} (r);
    \draw[->] (l) to[bend right=6] node[pos=.5,auto,swap] {$\scriptstyle \cat{forget}$} (d);
    \draw[right hook->] (d) to[bend right=6] node[pos=.5,auto,swap] {$\scriptstyle $} (r);
  \end{tikzpicture}
\end{equation*}
Here,
\begin{itemize}[wide]
\item The local $\aleph_1$-presentability of the category $\cat{C}^*$ of unital $C^*$-algebras follows from \cite[Theorem 3.28]{ar} via \cite[Theorem 2.4]{pr1}.

\item The local $\aleph_1$-presentability for $\cat{OSp}$, $\cat{OSys}$, and $\cat{OAlg}$ is provided by \cite[Theorem 4.35]{2412.20999v1}, \Cref{cor:osys.al1.pres} and \Cref{th:oalg.al1} respectively.

\item Finally, the left adjoint to the lower left-hand functor forms the object of \cite[Proposition 8]{zbMATH01199813}, with the monadicity claim handled in \Cref{cor:cast2osys.mndc} above. 
\end{itemize}

\begin{remarks}\label{res:no.mnd}
  \begin{enumerate}[(1),wide]
  \item\label{item:res:no.mnd:not.mnd} A few of the forgetful functors discussed above, in reference to function systems, are \emph{not} monadic. This is the case for the functor $U$ in \Cref{eq:osp.min}, which fails to reflect isomorphisms.

    Indeed, for normed spaces $(S,\|\cdot\|)$ of dimension $\ge 5$ the canonical morphism $\bullet\to \cat{min}(S)$ from the \emph{left} adjoint to $U$ applied to $S$ will fail to be a complete isometry \cite[Theorem 2.14]{zbMATH00152024}, even though its forgetful image is the identity on $(S,\|\cdot\|)\in \cat{Norm}$ and hence invertible.

  \item\label{item:res:no.mnd:uualg} On a related note, it is worthwhile to observe one qualitative distinction between the forgetful functors $\cat{OSp}^{\bullet}\xrightarrow{U}\cat{Norm}^{\bullet}$ and their multiplicative counterparts
    \begin{equation}\label{eq:ualg}
      \cat{OAlg}^{\bullet}
      \xrightarrow{\quad U_{alg}:=\cat{forget}\quad}
      \cat{NAlg}^{\bullet}
      ,\quad
      \bullet\in \left\{\text{blank},\ \wedge\right\}
    \end{equation}
    where $\cat{NAlg}$ denotes the category of unital normed algebras (equipped with a submulticaltive norm such that $\|1\|=1$) and $\cat{NAlg}^{\wedge}$ is the category of unital Banach algebras.
    
    The $\mathrm{max}$ construction of \cite[\S 3.3]{er_os_2000} provides the left adjoints to the two versions of $U$ (whether complete or not), and is fully faithful by its strict-quantization property \cite[\S 3.3, p.49]{er_os_2000}. Further, $U_{alg}$ also has left adjoints: the \emph{enveloping operator algebra} $\cO(A)$ \cite[\S 2.4.6]{bm_oa-mod} of a Banach algebra $A$ (in the complete case), which is also referred to as $\cat{MAXA}(\cA)$ in \cite[post Theorem 18.8]{pls-bk}. Those left adjoints, however, cannot be fully faithful as it is equivalent to the \emph{unit}
    \begin{equation*}
      \id
      \xrightarrow{\quad}
      U_{alg}\circ\left(\text{left adjoint}\right)
    \end{equation*}
    being a natural isomorphism \cite[Proposition 3.4.1]{brcx_hndbk-1}. However, the functor $U_{alg}$ is not even essentially surjective on objects. Indeed, a necessary condition for a Banach algebra to be an operator algebra (and hence, completely-isometrically embeddable in $\cL(H)$ for some Hilbert space $H$ \cite[Corollary 16.7]{pls-bk}) is that it satisfy the \emph{von Neumann inequality} \cite[Corollary 1.2]{pls-bk}:
    \begin{equation*}
      \forall\left(a\in A\right)
      \forall\left(p\in \bC[x]\right)
      \left(\|a\|\le 1 \xRightarrow{\quad} \|p(a)\|\le \sup_{\bD}|p|\right)
      ,\quad
      \bD:=\text{unit disk in }\bC.
    \end{equation*}
    Even finite-dimensional Banach algebras, generally, do not satisfy the von Neumann inequality: the endomorphism algebra $\cL(E)$ of a complex Banach space $E$, for instance, satisfies the von Neumann inequality precisely when $E$ is isometric to a Hilbert space \cite[Theorem 1.9]{pis_sim}. In fact, it is unknown whether the von Neumann inequality exactly defines those unital Banach algebras that may be given the structure of an operator algebra \cite[p.226]{pls-bk}.
    
  \item\label{item:res:no.mnd:mnd.lift} Consider the diagram
    \begin{equation}\label{eq:oalgsp.dmnd}
      \begin{tikzpicture}[>=stealth,auto,baseline=(current  bounding  box.center)]
        \path[anchor=base] 
        (0,0) node (l) {$\cat{OAlg}^{\bullet}$}
        +(2,.5) node (u) {$\cat{OSp}^{\bullet}$}
        +(2,-.5) node (d) {$\cat{NAlg}^{\bullet}$}
        +(4,0) node (r) {$\cat{Norm}^{\bullet}$}
        ;
        \draw[->] (l) to[bend left=6] node[pos=.5,auto] {$\scriptstyle $} (u);
        \draw[->] (u) to[bend left=6] node[pos=.5,auto] {$\scriptstyle $} (r);
        \draw[->] (l) to[bend right=6] node[pos=.5,auto,swap] {$\scriptstyle \ast$} (d);
        \draw[->] (d) to[bend right=6] node[pos=.5,auto,swap] {$\scriptstyle $} (r);
      \end{tikzpicture}
    \end{equation}
    of forgetful functors. We make the following observations:
    \begin{itemize}[wide]
    \item The upward arrows are monadic (\Cref{th:env.oalg} for the top and the bottom follows from \cite[Theorem 5.1]{zbMATH07469564}).

    \item The top downward arrow is, as discussed \cite[\S 3.3]{er_os_2000}, a right (as well as a left) adjoint.

    \item The leftmost corner is cocomplete (\Cref{th:oalg.al1}).
    \end{itemize}
    This is enough to ensure the existence of a left adjoint to the marked arrow by \emph{monadic lifting} \cite[Theorem 4.5.6]{brcx_hndbk-2}, giving an alternative take on the aforementioned operator-algebra envelope constructed in \cite[\S 2.4.6]{bm_oa-mod}.
  \end{enumerate}
\end{remarks}

We build on \Cref{res:no.mnd}\Cref{item:res:no.mnd:uualg} by raising another fundamental distinction between the functors $U$ and $U_{alg}$. Although the former are \emph{left} adjoints as well as right (with the minimal $\cat{min}$ functor being the right adjoints), the latter cannot be.

\begin{proposition}\label{pr:coeq.not.coprod}
  The forgetful functors \Cref{eq:ualg} preserve coequalizers but not arbitrary (binary) coproducts. In particular, said functors are not left adjoints. 
\end{proposition}
\begin{proof}
  For simplicity, we focus on the complete categories. The coequalizer of a pair $A\xrightarrow{f,g}B$ in $\cat{OAlg}^{\wedge}$ is the operator-space quotient \cite[\S 3.1]{er_os_2000} by the closed ideal generated by all $f(a)-g(a)$, $a\in A$. It inherits an operator-algebra structure by the projectivity of the Haagerup tensor product, much as in the proof of \Cref{th:env.oalg}. 

  As to the failure to preserve (even binary) coproducts, it suffices to consider one such example: the coproduct of two copies of $C(\bD)$, the continuous complex-valued function algebra on the unit disk $\bD$. Upon decorating coproducts by the respective categories that house them, we have a canonical morphism
  \begin{equation*}
    C(\bD)\coprod^{\cat{NAlg}^{\wedge}}C(\bD)
    \xrightarrow{\quad}
    C(\bD)\coprod^{\cat{OAlg}^{\wedge}}C(\bD)
    \quad\text{in}\quad
    \cat{NAlg}^{\wedge}
  \end{equation*}
  which we claim cannot be onto. Indeed, if it were so, it would induce a surjection between the \emph{abelianizations} of the two Banach algebras. However, that morphism coincides with the canonical morphism
  \begin{equation}\label{eq:cdcd}
    C(\bD)\widehat{\otimes}_{\pi} C(\bD)
    \xrightarrow{\quad}
    C(\bD)\widehat{\otimes}_{\varepsilon} C(\bD)
    \overset{\text{\cite[Example 4.2(3)]{df_tens-id}}}{\cong}
    C(\bD\times \bD)
  \end{equation}
  from the \emph{projective} \cite[\S 2.1]{ryan_ban} to the \emph{injective} \cite[\S 3.1]{ryan_ban} Banach-space tensor product. Indeed, the maximal Banach tensor product of Banach algebras is again naturally a Banach algebra \cite[Theorem 1]{zbMATH03156609} and hence, the universal recipient of two mutually-commuting maps. On the other hand, the identification
  \begin{equation*}
    \left(C(\bD)\coprod^{\cat{OAlg}^{\wedge}}C(\bD)\right)_{ab}
    \cong
    C(\bD\times \bD)
  \end{equation*}
  is noted in somewhat different phrasing on \cite[p.261]{pls-bk} (where the free operator algebra on two commuting contractions is identified with the polynomial algebra on $\bD^2$).

  The conclusion now follows from the non-surjectivity of \Cref{eq:cdcd}, which is not a linear-topological isomorphism by \cite[\S II.2.6, Proposition 2.50]{hlmsk_homolog}. In fact, since it is an injective map between Banach spaces \cite[\S 12.10]{df_tens-id}, it would be an isomorphism if it were surjective \cite[Theorem III.12.5]{conw_fa}. 
\end{proof}

Given the partial monadicity of \Cref{eq:oalgsp.dmnd}, noted in \Cref{res:no.mnd}\Cref{item:res:no.mnd:mnd.lift} above, a natural question that is very much in the same circle of ideas is whether the asterisk-marked morphisms in that diagram are monadic. That they are not is another bit of misbehavior, in the spirit of \Cref{pr:coeq.not.coprod}, which disappears in the non-multiplicative setting. A brief reminder will help to make sense of the statement. 

Recall \cite[\S 3.2, post Proposition 2]{bw} that an adjunction
\begin{equation*}
  \begin{tikzpicture}[>=stealth,auto,baseline=(current  bounding  box.center)]
    \path[anchor=base] 
    (0,0) node (l) {$\cC$}
    +(4,0) node (r) {$\cD$}
    +(2,0) node () {$\bot$}
    ;
    \draw[->] (l) to[bend left=20] node[pos=.5,auto] {$\scriptstyle F$} (r);
    \draw[->] (r) to[bend left=20] node[pos=.5,auto] {$\scriptstyle U$} (l);
  \end{tikzpicture}
\end{equation*}
induces a canonical \emph{comparison functor} $\cD\to \cC^{UF}$. Monadicity (for $U$) means precisely that the comparison functor is equivalent, while $U$ is sometimes (e.g. \cite[p.954]{pr_tc-1}) said to be \emph{pre-monadic} if the comparison functor is fully faithful. 

\begin{theorem}\label{th:not.pre.mnd}
  Suppose $A\in \cat{OAlg}^{\bullet}$ admits a contractive morphism $A\xrightarrow{\rho} \cL(H)$ which is not completely contractive for some Hilbert space $H$.
  \begin{enumerate}[(1),wide]
  \item\label{item:th:not.pre.mnd:counit} Denoting by $F$ the left adjoint to the forgetful functor $U_{alg}$ of \Cref{eq:ualg}, the counit $FU(A)\to A$ is bijective but not completely isometric, so in particular not a regular epimorphism in $\cat{OAlg}$.

  \item\label{item:th:not.pre.mnd:iso.refl} $A$ is the codomain of an isometry in $\cat{OAlg}$ that is not a complete isometry.

  \item\label{item:th:not.pre.mnd:wtns} In particular, such an object $A$ witnesses the failure of both pre-monadicity and monadicity for $U_{alg}$. 
  \end{enumerate}
\end{theorem}
\begin{proof}
  \begin{enumerate}[label={},wide]
  \item\textbf{\Cref{item:th:not.pre.mnd:wtns}} Recall that pre-monadicity is equivalent \cite[\S 3.3, Theorem 9]{bw} to all counits being \emph{regular epimorphisms}, i.e. \cite[\S 1.7, Exercise (REGMON)$^{\circ}$]{bw} coequalizers of parallel arrow pairs. Thus, pre-monadicity failure follows from \Cref{item:th:not.pre.mnd:counit}.

    This surely rules out monadicity (it being the stronger property), but \Cref{item:th:not.pre.mnd:iso.refl} also achieves this directly. Indeed, monadicity requires \cite[\S 3.3, Theorem 10]{bw} isomorphism reflection and a morphism in $\cat{OAlg}^{\bullet}$ that is isometric but not completely isometric is precisely a non-isomorphism that becomes an isomorphism upon forgetting to $\cat{NAlg}^{\bullet}$.

  \item\textbf{\Cref{item:th:not.pre.mnd:counit}} The last claim follows from the rest: regular epimorphisms in $\cat{OAlg}$ are quotients in the sense of \cite[\S 2.4]{pis_os}, so a non-isomorphic bijection cannot be one. 

    As to the principal claim, it follows from the description of $F$ given in \cite[post Theorem 18.8]{pls-bk} (where the functor is denoted by $\cat{MAXA(-)}$). Indeed, $B\to UF(B)$ is bijective as soon as $B$ embeds isometrically into an operator algebra and hence, this implies the bijectivity of $FU(A)\xrightarrow{\varepsilon} A$ for every $A\in \cat{OAlg}$. On the other hand, our assumption on $A$ means that, for some $n$, the map $\varepsilon\otimes \id_{M_n}$ cannot be isometric.
    
  \item\textbf{\Cref{item:th:not.pre.mnd:iso.refl}} This is effectively what \cite[discussion pre Proposition 7.9]{pls-bk} carries out for the polynomial algebra on three complex-disk variables.

    First note that $A$, being an operator algebra, admits a completely isometric embedding $A\xrightarrow{\iota}\cL(K)$ for some Hilbert space $K$. The representation $A\xrightarrow{\rho\oplus\iota}\cL(H\oplus K)$ will thus be isometric and matricial-norm-non-increasing, strictly so in at least one instance. It follows that
    \begin{equation*}
      \left(
        \left(\rho\oplus\iota\right)(A)
        \xrightarrow{\quad\left(\rho\oplus\iota\right)^{-1}\quad}
        A
      \right)
      \quad
      \in
      \quad
      \cat{OAlg}
    \end{equation*}
    is an isometric non-isomorphism. 
  \end{enumerate}
\end{proof}

Consequently, we have the following.

\begin{corollary}\label{cor:alg.not.pre.mnd}
  The forgetful functors \Cref{eq:ualg} are not pre-monadic and, in particular, are not monadic. 
\end{corollary}
\begin{proof}
  Focusing on the incomplete branch to fix the language, a celebrated example due to Parrott: \cite[pp.89-91]{pls-bk}, \cite[Example 25.18]{pis_os} shows that the polynomial algebra $P(\bD^n)$ on $n$-variables, $n\geq 3$, in the unit disk $\bD\subset \bC$, equipped with its operator-algebra structure inherited from $P(\bD^n)\le C(\bT^n)$, satisfies the hypothesis of \Cref{th:not.pre.mnd}. 
\end{proof}

\def\polhk#1{\setbox0=\hbox{#1}{\ooalign{\hidewidth
  \lower1.5ex\hbox{`}\hidewidth\crcr\unhbox0}}}

\addcontentsline{toc}{section}{References}

\Addresses

\end{document}